\documentclass[12pt]{article}
\usepackage{amsmath}
\usepackage{amssymb}
\usepackage{amsthm}
\usepackage{paralist}
\usepackage{mathabx}
\usepackage[utf8]{inputenc}
\usepackage{hyperref}
\usepackage{mathrsfs}
\usepackage{csquotes}
\usepackage{lipsum}
\usepackage{eufrak}
\usepackage[english]{babel}

\usepackage{mathtools}

\DeclarePairedDelimiter\floor{\lfloor}{\rfloor}
\usepackage{amsmath,mathrsfs,amsfonts,amssymb,amsthm,
amstext,pgf,graphicx,hyperref,verbatim,lmodern,textcomp,
color,young,tikz,enumerate,mathptmx,calc,mathtools,tikz,pgf,
blkarray,bigstrut,xcolor}
\usepackage{fullpage}
\usepackage[mathscr]{euscript}

\newtheorem{theorem}{Theorem}[section]

\newtheorem{conjecture}[theorem]{Conjecture}
\newtheorem{lemma}[theorem]{Lemma}

\newtheorem{remark}[theorem]{Remark}

\begin{document}

\title{On the eigenvalues of Grassmann graphs, Bilinear forms graphs and Hermitian forms graphs}
\author{Sebastian M. Cioab\u{a}\footnote{Department of Mathematical Sciences, University of Delaware, Newark, DE 19716-2553, {\tt cioaba@udel.edu}} \, and Himanshu Gupta\footnote{Department of Mathematical Sciences, University of Delaware, Newark, DE 19716-2553, {\tt himanshu@udel.edu}}}
\date{\today}
\maketitle

\begin{abstract}
Recently, Brouwer, Cioab\u{a}, Ihringer and McGinnis obtained some new results involving the eigenvalues of various graphs coming from association schemes and posed some conjectures related to the eigenvalues of Grassmann graphs, bilinear forms graphs and Hermitian forms graphs. In this paper, we prove some of their conjectures.    
\end{abstract}

\section{Introduction}

The eigenvalues of graphs are closely related to important combinatorial parameters and play important roles in many situations (see \cite{BH,GR} for example). In particular, the smallest eigenvalue of a graph is related to its independence number and the size of its largest bipartite subgraph, also known as the max-cut of the graph. The connection to the independence number has been exploited to give algebraic proofs of many Erd\H{o}s-Ko-Rado type problems (see \cite{GM}). The interactions between the smallest eigenvalue and the max-cut is the basis of the famous Goemans-Williamson semidefinite programming approximation algorithm whose performance ratio is at least $\alpha=\frac{2}{\pi}\min_{0<\theta\leq \pi}\frac{\theta}{1-\cos\theta}$ ($0.87856<\alpha<0.87857$) (see \cite{GW}). Karloff \cite{K} proved that the performance ratio of the Goemans-Williamson SDP approximation algorithm is exactly $\alpha$. The crucial part of Karloff's argument was determining the smallest eigenvalue of some graphs in the Johnson association scheme. Karloff made a conjecture about the smallest eigenvalue of Johnson graphs in a wider range of parameters; this conjecture was recently proved in \cite{BCIM}. Goemans and Williamson \cite{GW} also proved that the performance ratio of their algorithm can be improved for graphs that are close to being bipartite. Alon and Sudakov \cite{AS} (see also Alon, Sudakov and Zwick \cite{ASZ}) extended Karloff's work and showed that the performance ratio of this extended Goemans-Williamson algorithm is best possible. Again, the key part of Alon and Sudakov's work was determining the smallest eigenvalue of some graphs in the Hamming association scheme. The smallest eigenvalue of graphs in the Hamming scheme was determined for various ranges in \cite{AS,vDS,DK}. Van Dam and Sotirov \cite{vDS} posed a conjecture regarding the smallest eigenvalue of some graphs in the Hamming scheme; this was proved in the binary case in \cite{DK} and in full generality in \cite{BCIM}. In addition, the paper \cite{BCIM} investigated the eigenvalues of various graphs in other association schemes such as the Grassmann scheme, the bilinear forms scheme and the Hermitian forms scheme. In this paper, we extend that work and prove some of the conjectures from \cite{BCIM} regarding the eigenvalues of these graphs. We describe the results in more details in the following subsections.

\subsection{Grassmann graphs}

Let $V$ be a vector space of dimension $n$ over the field $\mathbb{F}_q$. The vertices of the Grassmann scheme $G_q(n,d)$ are the $d$-dimensional subspaces of $V$. Any two subspaces are $j$-related in the scheme if and only if their intersection has dimension $d-j$ for $0 \leq j \leq d$. Since $G_q(n,d)$ is isomorphic to $G_q(n,n-d)$ by an isomorphism that maps a subspace to its orthogonal complement so we assume that $n \geq 2d$. We denote the graphs of the Grassmann scheme by $G_q(n,d,j)$ for $0 \leq j \leq d$. The eigenvalues of $G_q(n,d,j)$ are $P_{ij} = G_j(i)$ $(0 \leq i \leq d)$, where 
\begin{align}
 G_j(i) =& \sum_{h=0}^{j}(-1)^{j-h}q^{hi+\binom{j-h}{2}} { d-i \brack h}{d-h \brack j-h}{ n-d-i+h \brack h} \label{eq:firstexpressiongrassman} \\ 
 =& \sum_{h=0}^{i}(-1)^{i-h}q^{j(j-i+h)+\binom{i-h}{2}} {i \brack h}{d-h \brack j}{ n-d-i+h \brack n-d-j} \label{eq:secondexpressiongrassman}
 \end{align}
(see Delsarte \cite{Del}, Theorem 10, and Eisfeld \cite{Es}, Theorem 2.7). 

The following conjecture was posed in \cite[Conjecture 5.5]{BCIM}.

%----------------------------------------------
\begin{conjecture}\label{ConjGrassmann}
\begin{enumerate}[(i)]
\item If $(n,q) \neq (2d,2)$, then $|G_j(i+1)| < |G_j(i)|$ where $0 \leq i \leq d-1$.
\item If $(n,q) = (2d,2)$, then $G_j(d-j)$ is negative for $(d,j) = (5,3)$ and when $d \geq 6$, $2 \leq j \leq d-2$, and $G_j(d-j)$ is the smallest among the $G_j(i)$ when $d \geq 6$, $3 \leq  j \leq d-2$.
\end{enumerate}
\end{conjecture}
%----------------------------------------------

Part (i) has been proved in \cite{BCIM} for $q \geq 5$, see \cite{BCIM}. In this paper, we extend that proof for $q \geq 3$ in Theorem \ref{q>2} and for $q=2$ and $n\geq 2d+1$ in Theorem \ref{q=2}. Part (ii) has been proved in \cite[Theorem 5.8]{BCIM} for $7 \leq j \leq d-5$ and it is still open for $j\in\{2,3,4,5,6,d-4,d-3,d-2\}$.

\subsection{Bilinear forms graphs}

The vertices of the Bilinear forms scheme $H_q(d, e)$ are all the $d\times e$ matrices over the field $\mathbb{F}_q$, where $d\leq e$. Any two matrices are $j$-related in the scheme if and only if their difference has rank $j$ for $0\leq j \leq d$. We denote the graphs of the Bilinear forms scheme by $H_q(d,e,j)$ for $0\leq j \leq d$. The eigenvalues of $H_q(d,e,j)$ are $P_{ij} = B_j(i)$ ($0 \leq i \leq d$), where
$$
B_j(i) = \sum_{h=0}^j (-1)^{j-h}q^{eh+\binom{j-h}{2}} { d-h \brack d-j}{d-i \brack h}
$$
(see Delsarte \cite{DelB}, Theorem A2).

The following conjecture was posed in \cite[Conjecture 7.6]{BCIM}.

%--------------------Conjecture1--------------------------
\begin{conjecture}\label{Bilinear3.1}
For $q \geq 3$, or $q = 2$ and $d \neq e$, $B_j(d - j + 1)$ is the smallest
eigenvalue in the distance-$j$ graph for $1 \leq  j \leq d$. 
\end{conjecture}
This conjecture was proved for $q\geq 4$ in \cite{BCIM} (see Proposition 7.8). Here we complete its proof for $q\geq 2$ in Theorem \ref{Bilinear3.8}.

\subsection{Hermitian forms graphs}

The vertices of the Hermitian forms scheme $Q_q(d)$ are all the $d\times d$ Hermitian matrices over $\mathbb{F}_{q^2}$. Any two Hermitian matrices are $j$-related in the scheme if and only if their difference has rank $j$ for $0\leq j \leq d$. We denote the graphs of the Hermitian forms scheme by $Q_q(d,j)$ for $0\leq j \leq d$. The eigenvalues of $Q_q(d,j)$ are $P_{ij} = Q_j(i)$ ($0 \leq i \leq d$), where
$$Q_j(i) = (-1)^j \sum_{h=0}^{j} (-q)^{\binom{j-h}{2}+hd}{d-h \brack d-j}_b{d-i \brack h}_b$$ (see Schmidt \cite{Schmidt} and Stanton \cite{Stanton}). Here the Gaussian coefficients have base $b = -q$, that is, for any non-negative integers $m$ and $l$ we have ${m \brack l}_b = \prod_{i=1}^{l}\frac{(-q)^{m-i+1}-1}{(-q)^i-1}$. Note that its sign is $(-1)^{(m+1)l}$.

The following conjectures were posed in \cite{BCIM} (see Conjecture 9.2 and 9.3).
\begin{conjecture}\label{conj1}
 \begin{enumerate}[(i)]
 \item If $j$ is odd, then $Q_j(1) \leq Q_j(i)$ for $0 \leq i \leq d$.
 \item If $j$ is even, $j\geq 2$, then $Q_j(d-j+2) \leq Q_j(i)$ for $0\leq i \leq d$.
 \end{enumerate}
\end{conjecture}

\begin{conjecture}\label{conj2}
Let $d\geq 3$. Then $|Q_j(i)| < |Q_j(1)|$ for $2 \leq i \leq d$. 
\end{conjecture}

These conjectures were proved for $q \geq 4$ in \cite{BCIM} (See Theorem 9.5). Here we prove these conjectures for $q\geq 2$ in Theorem \ref{HermitianTheorem}. We assume $d \geq 6$ in their proof. However, if $2\leq d\leq 5$, then for $q \geq 4$ we refer to Theorem 9.5 in \cite{BCIM}, and for $q =2,3$ one can easily check by computing precisely eight eigenmatrices\footnote{The details of these calculations can be found on the webpage\\ {\tt https://github.com/Himanshugupta23/Hermitian-Graph-Eigenmatrix}}. This completely settles the above conjectures.

\section{Some useful lemmas}
The following lemma is useful while estimating the absolute value of an alternating series by its first and second dominating terms. 
%---------------------------------------
\begin{lemma}\label{mainlemma}
Let ${A = \sum_{i=0}^n  a_i}$ be an alternating series with terms increasing in absolute value, from $i = 0$ to $n$. Then $A$ has the same sign as that of $a_n$, and $|a_n|-|a_{n-1}| \leq |A| \leq |a_n|.$
\end{lemma}
\begin{proof}
We give a proof by induction on $n$. Clearly, the statement holds true for $n =1$. Let us assume that it is true for $n =k$. That is, assume that the statement holds true for any alternating series with terms increasing in absolute value, from $i = 0$ to $k$. We want to prove for $n =k+1$. Consider ${A = \sum_{i=0}^{k+1}  a_i}$, an alternating series with terms increasing in absolute value, from $i = 0$ to $k+1$. Let $A' = \sum_{i=0}^k a_i$, thus, $A = A'+a_{k+1}$. By the induction hypothesis $A'$ has the same sign as that of $a_k$, so opposite as that of $a_{k+1}$, and $|a_k|-|a_{k-1}| \leq |A'| \leq |a_k|.$ Hence, $A$ has the same sign as that of $a_{k+1}$, and $|a_{k+1}|-|a_k| \leq |A| \leq |a_{k+1}|-|a_{k}|+|a_{k-1}| \leq |a_{k+1}|$.
\end{proof}
 %---------------------------------------
 
The inequalities in the following lemma are useful while estimating the Gaussian coefficients and their quotients. Proofs are given by F.\ Ihringer and K.\ Metsch in \cite{Ih} (see Section 6). We include the less technical proofs here and refer to \cite{Ih} for not hard but more technical proofs. 
%---------------------------------------
\begin{lemma} \label{lem1}  \label{estimatelemma}
\begin{enumerate}[(i)]
\item If $0\leq k \leq n$ and $q\geq 2$, then $\displaystyle{{n \brack k}_q \geq q^{k(n-k)}}$.
\item If $0 < k < n$ and $q\geq 2$, then $\displaystyle{{n \brack k}_q \geq \left(1 + \frac{1}{q}\right) q^{k(n-k)}}$.
\item If $0 \leq k \leq n$ and $q \geq 3$, then $\displaystyle{{n \brack k}_q < 2 q^{k(n-k)}}$.
\item If $k =0, 1, n-1,$ or $n$, and $q \geq 2$, then $\displaystyle{{n \brack k}_q < \left(\frac{q}{q-1}\right)q^{k(n-k)}}$. $\qed$
\end{enumerate} 
\end{lemma}
\begin{proof}
\begin{enumerate}[(i)]
    \item If $0\leq k \leq n$ and $q\geq 2$, then
    $$
    {n \brack k}_q = \prod_{i=1}^k \frac{q^{n-k+i}-1}{q^i-1} \geq \prod_{i=1}^k \frac{q^{n-k}(q^i-q^{-(n-k)})}{q^i-1} \geq \prod_{i=1}^k q^{n-k} = q^{k(n-k)}.
    $$
    \item If $0 < k < n$ and $q\geq 2$, then
    $$
    {n \brack k}_q = \frac{q^{n-k+1}-1}{q-1}\prod_{i=2}^k \frac{q^{n-k+i}-1}{q^i-1} \geq \frac{q^{n-k}(q^2-q^{-(n-k-1)})}{q(q-1)}\cdot q^{(k-1)(n-k)} \geq \left(1+\frac{1}{q}\right)q^{k(n-k)}.
    $$
    \item See Lemma 34 (a) in \cite{Ih} for proof. 
    \item If $k = 1$ or $n-1$ and $q\geq 2$, then 
    $$
    \displaystyle{{n \brack k}_q = \frac{q^n-1}{q-1} \leq \frac{q^n}{q-1} = \left(\frac{q}{q-1}\right)q^{k(n-k)}}.
    $$
    If $k=0$ or $n$ and $q\geq 2$, then
    $\displaystyle{
    {n \brack k}_q = 1 < \left(\frac{q}{q-1}\right)q^{k(n-k)}.}$ \qedhere
\end{enumerate}
\end{proof}
%---------------------------------------

\section{Proof of Conjecture \ref{ConjGrassmann} (i)}

The term corresponds to the index $h$ such that either $h> \min\{i,d-j\}$, or $h < \max\{0,i-j\}$ in (\ref{eq:secondexpressiongrassman}) is zero. Therefore, we write the expression as
\begin{align}
 G_j(i) = \sum_{h=h_{\min}(i,j)}^{h_{\max}(i,j)}(-1)^{i-h}q^{j(j-i+h)+\binom{i-h}{2}} {i \brack h}{d-h \brack j}{ n-d-i+h \brack n-d-j} \label{eq:thirdexpressiongrassman}
\end{align}
 where $h_{\max}(i,j) := \min\{i,d-j\}$ and $h_{\min}(i,j) := \max\{0,i-j\}$. Whenever $i, j$ are clear by the context we just write $h_{\max}$ and $h_{\min}$. We mostly consider the expression (\ref{eq:thirdexpressiongrassman}) for $G_j(i)$. Let us denote the $h$-th term of this expression by $T_h(i,j)$. Whenever $i,j$ are clear by the context we just write $T_h$. 
 
Let $g_h(i,j)$ (if $i,j$ are clear by context we just write $g_h$) be the exponent of $q$ in $T_h(i,j)$ if we approximate ${n \brack k}$ with $\displaystyle{q^{k(n-k)}}$, that is,
\begin{align*}
g_{h}(i,j) &= j(j-i+h) + \binom{i-h}{2} + h(i-h) + j(d-h-j) + (n-d-j)(h-i+j)\\
&= -\frac{1}{2}h^2+h\left(n-d-j+\frac{1}{2}\right) + j(d-j)+\frac{i(i-1)}{2}+(n-d)(j-i).
\end{align*}
Let $h_0 = n-d-j+\frac{1}{2}$. Then the quadratic expression $g_h(i,j)$ is maximal for $h = h_0$ and $g_{h_0+x} = g_{h_0}-\frac{1}{2}x^2$. The terms occurring in the sum have indices $h$ with $h \leq h_{\max}< h_0$, so the term with largest index has largest exponent. 
%----------------------------------------------------------------------------------
\begin{lemma} \label{BoundsGrassman}
If $(n,q) \neq (2d,2)$, then $|T_{h_{\max}}(i,j)| - |T_{h_{\max}-1}(i,j)| \leq |G_j(i)| \leq |T_{h_{\max}}(i,j)|$, where $T_{-1}(i,j) = 0$. 
\end{lemma}
\begin{proof}
If $h_{\max} = 0$, then it holds true. Let $h_{\max} \geq 1$. Then $1\leq i \leq d$ and $1\leq j \leq d-1$. We note two inequalities. First, $g_{h} - g_{h-1} = n-d-j-h+1 \geq 2$ for all $h_{\min} +1 \leq h \leq h_{\max}$ except if $n=2d$ and $h = d-j$. Second, $q^{g_{h}} < |T_h| < 2^3 q^{g_{h}}$ for all $q\geq 3$ by using Lemma \ref{estimatelemma} (i) and (iii). These two inequalities imply that $\displaystyle{\frac{|T_h|}{|T_{h-1}|} > \frac{q^{g_{h} - g_{h-1}}}{8} \geq \frac{q^2}{8 }> 1}$ for all $h_{\min} + 1 \leq h \leq h_{\max}$ and $q\geq 3$ except if $n=2d$ and $h = d-j$. However, if $q\geq 3$, $n =2d$, and $h = d-j$, then $d\leq i+j \leq 2d-1$. Since $d\geq 2$, we get that
\begin{align}\label{eq:n=2d}
\frac{|T_{d-j-1}|}{|T_{d-j}|}
= \frac{q^{i-d}(q^{d-i}-1)(q^{d-j}-1)(q^{j+1}-1)}{(q-1)(q^{(i+j)-d+1}-1)(q^{2d-(i+j)}-1)} 
< \frac{q^{d+1}}{(q-1)^2(q^{d}-1)}
\leq \frac{27}{32} < 1.
\end{align}
On the other hand, if $q=2$ and $n \geq 2d+1$, then $n-d-j-h \geq 2$ except if $n = 2d+1$ and $h =d-j$. Thus, for all $h_{\text{min}} +1 \leq h \leq h_{\text{max}}$ and $q=2$ we obtain that 
$$
\frac{|T_{h-1}|}{|T_h|} = \frac{2^{-j+i-h}(2^h-1)(2^{d-h+1}-1)(2^{j-i+h}-1)}{(2^{i-h+1}-1)(2^{d-h+1-j}-1)(2^{n-d-i+h}-1)}
< \frac{2^{d+1}}{2^{n-j-h-1}}
= \frac{1}{2^{n-d-j-h-2}} \leq 1$$
except if $n = 2d+1$ and $h =d-j$. However, if $q=2$, $n = 2d+1$ and $h =d-j$, then $d\leq i+j \leq 2d-1$. Thus, we conclude that
$$
\frac{|T_{d-j-1}|}{|T_{d-j}|}
= \frac{(1-2^{-(d-i)})(2^{d-j}-1)(2^{j+1}-1)}{(2^{(i+j)-d+1}-1)(2^{2d-(i+j)+1}-1)} 
< \frac{2^{d+1} + 1 -(2^{d-{\floor*{\frac{d-1}{2}}}}+2^{{\floor*{\frac{d-1}{2}}}+1})}{(2-1)(2^{d+1}-1)}
< 1.
$$
Therefore, if $(n,q) \neq (2d,2)$, the expression for $G_j(i)$ is an alternating series with terms increasing in absolute value, from $h = h_{\min}$ to $h_{\max}$. Hence, Lemma \ref{mainlemma} imply the assertion.
\end{proof}

%-----------------1.4---------------------
\begin{lemma}\label{Intermediatelemma}
Let $q\geq 3$, $0 \leq i \leq d$, $1\leq j \leq d$, and $c := g_{h_{\max}}(i,j)$. Then $\displaystyle{\frac{4}{9}q^c < |G_{j}(i)| < 4q^c}$ except if $n = 2d$ and $d-j \leq i$. However, if $n =2d$ and $d-j \leq i$, then $\displaystyle{1 \geq \frac{|G_j(i)|}{|T_{d-j}(i,j)|} \geq \frac{5}{32}}$.
\end{lemma}
\begin{proof}
Let $h = h_{\max}$. Then either ${i \brack h}$ or ${d-h \brack j}$ is equal to $1$. By using Lemma \ref{estimatelemma} (iii) and Lemma \ref{BoundsGrassman} we obtain the upper bound $|G_j(i)| \leq |T_{h}| < 2^2q^c = 4q^c.$ For the lower bound, if $h=0$, then we are done by Lemma \ref{estimatelemma} (i). If $h\geq 1$, then there is at least one Gaussian coefficient in $T_{h}$ that is not equal to 1. Also, $g_{h} - g_{h-1} = n-d-j-h+1 \geq 2$ except if $n=2d$ and $h = d-j$. Thus, by using Lemma \ref{BoundsGrassman}, Lemma \ref{estimatelemma} (i), (ii), and (iii) we obtain the lower bound
$$\displaystyle{|G_j(i)| \geq |T_{h}| - |T_{h - 1}|> \left(1+\frac{1}{q}\right)q^c - 2^3 q^{c-2} \geq \frac{4}{9}q^c}.$$
If $n =2d$ and $d-j \leq i$, then $h= d-j$. Thus, by using Lemma \ref{BoundsGrassman} and equation (\ref{eq:n=2d}) we get that
$$
1 \geq \frac{|G_j(i)|}{|T_{d-j}(i,j)|} \geq 1 - \frac{|T_{d-j-1}(i,j)|}{|T_{d-j}(i,j)|} \geq 1-\frac{27}{32} = \frac{5}{32}.
$$
Hence, we show the assertions. 
\end{proof}
%--------------------------------------------------
\begin{remark}\label{remarkgrassman}
If $j=1$ and $i< d-1$, we need to consider a better lower bound for $|G_{j}(i)|$ than given in above Lemma \ref{Intermediatelemma}. That is, 
$$|G_{j}(i)| \geq |T_{h}| - |T_{h - 1}| > \left(1+\frac{1}{q}\right)^2 q^c- 2^2q^{c-2} \geq \frac{12}{9}q^c.$$
\end{remark}
We can now prove the Conjecture \ref{ConjGrassmann} (i) for $q\geq 3$. 
%------------------1.5----------------------------------
\begin{theorem}\label{q>2}
Let $q \geq 3$, $0\leq i \leq d-1$, and $1\leq j \leq d$. Then $|G_j(i+1)| < |G_j(i)|.$
\end{theorem}
\begin{proof}
Let $h := h_{\max}(i,j)$, $h' := h_{\max}(i+1,j)$, $c := g_h(i,j)$, and $c' := g_{h'}(i+1,j)$. If $d-j \leq i$, then $c-c' = n-d-i$ and if $i < d-j$, then $c-c' = j$. We break the proof into three cases. First, let $(n,h) \neq (2d,d-j)$. Then $c-c' \geq 2$ except if $j=1$ and $i < d-1$.
Thus, by using Lemma \ref{Intermediatelemma} we obtain that $\displaystyle{\frac{|G_j(i)|}{|G_j(i+1)|} >  \frac{\frac{4}{9}q^2}{4} \geq 1}$ except if $j=1$ and $i < d-1$. However, if $j=1$ and $i< d-1$, 
then $c-c' = 1$. Thus, Lemma \ref{Intermediatelemma} and Remark \ref{remarkgrassman} imply that $\displaystyle{\frac{|G_j(i)|}{|G_j(i+1)|} > \frac{\frac{12}{9}q}{4} \geq 1.}$

Second, let $(n,h) = (2d,d-j)$ and $0\leq i \leq d-2$. Then $h' = d-j$, $d\geq 3$, and $d \leq i+j \leq 2d-2$. Thus,  
$$
\frac{|T_{d-j}(i,j)|}{|T_{d-j}(i+1,j)|} = \frac{(q^{(i+j)-d+1}-1)(q^{2d-(i+j)}-1   )}{q^{-(d-i)}(q^{i+1}-1)(q^{d-i}-1)} > \frac{(q-1)(q^{d}-1)}{q^-2(q^{d-1}-1)(q^2-1)} > \frac{q^3}{q+1} \geq \frac{27}{4}.
$$
Hence, by the second assertion of Lemma \ref{Intermediatelemma} we obtain that
$$
\frac{|G_j(i)|}{|G_j(i+1)|} \geq \frac{5}{32}\cdot\frac{|T_{d-j}(i,j)|}{|T_{d-j}(i+1,j)|} > \frac{5\cdot 27}{32\cdot 4} > 1.
$$
Lastly, let $(n,h) = (2d,d-j)$ and $i = d-1$. In this case we use (\ref{eq:firstexpressiongrassman}) to get that
$$
\frac{|G_j(d-1)|}{|G_j(d)|} \geq \frac{q^d(q^j-1)(q+1)}{q^j(q^d-1)}-1 \geq  (1-q^{-j})(q+1)-1 \geq (1-q^{-1})(q+1)-1 \geq \frac{5}{3} > 1.
$$
Thus, we complete the proof. 
\end{proof}

%---------------------1.6----------------
\begin{lemma}\label{Lemma0.8}
Let $q=2$, $n\geq 2d+1$, $0\leq i \leq d-1$, $1\leq j \leq d$, and $h:= h_{\max}(i,j)$. Then $$\displaystyle{1\geq \frac{|G_j(i)|}{|T_h(i,j)|}} \geq k,\ \text{where,}$$ 
\begin{enumerate}[(i)]
\item $k = 3/7$, if $(n,h) \neq (2d+1,d-j)$ and $2\leq j\leq d$;
\item $k = 5/7$, if $(n,h) \neq (2d+1,d-j)$ and $j=1$;
\item $k = 5/21$, if $n = 2d+1$ and $d-j < i$;
\item $\displaystyle{k = \frac{(2^j-1)(2^{d-j+1}-1)+(2^{d-j}-1)(2^j+1)+2^{j+1}(2^j-1)+1}{2^j(2^{d+1}-1)}}$, if $n = 2d+1$ and $d-j = i$.
\end{enumerate}
\end{lemma}
\begin{proof}
If $h = 0$, then assertion follows since $\displaystyle{\frac{|G_j(i)|}{|T_h|} = 1}$. We assume $h \geq 1$. Then $1\leq i \leq d-1$ and $1\leq j \leq d-1$ and by Lemma \ref{BoundsGrassman} we have that $\displaystyle{1 \geq \frac{|G_j(i)|}{|T_h|} \geq 1-\frac{|T_{h-1}|}{|T_h|}}$. Hence, we prove the assertion by finding the upper bounds for
$$\frac{|T_{h-1}|}{|T_h|} = \frac{2^{-j+i-h}(2^h-1)(2^{d-h+1}-1)(2^{j-i+h}-1)}{(2^{i-h+1}-1)(2^{d-h+1-j}-1)(2^{n-d-i+h}-1)}.$$
If $(n,h) \neq (2d+1,d-j)$, then $n-d-i+h \geq 3$ and $n-d-j-h \geq 2$.
\begin{enumerate}[(i)]
\item If $(n,h) \neq (2d+1,d-j)$ and $2\leq j\leq d$, then
$$\frac{|T_{h-1}|}{|T_h|} \leq  \frac{2^{-j+i- h} \cdot 2^{h}\cdot 2^{d-h+1}\cdot 2^{j-i+h}}{2^{i-h}\cdot 2^{d-h-j} \cdot 2^{n-d-i+h-1} \cdot 7/4} \leq \frac{4}{7(2^{n-d-j-h-2})} \leq \frac{4}{7}.$$

\item If $(n,h) \neq (2d+1,d-j)$ and $j=1$, then 
$$\frac{|T_{h-1}|}{|T_h|} \leq \frac{2^{-1} \cdot 2^{h}\cdot 2^{d-h+1}}{2^{i-h}\cdot 2^{d-h-j} \cdot 2^{n-d-i+h-1} \cdot 7/4} \leq
\frac{2}{7(2^{n-d-j-h-2})} \leq\frac{2}{7}.$$

\item If $n = 2d+1$ and $d-j < i$, then $d+1\leq i+j \leq 2d-2$ and $d\geq 3$. So,
$$
\frac{|T_{h-1}|}{|T_h|} 
= \frac{2^{i-d}(2^{d-j}-1)(2^{j+1}-1)(2^{d-i}-1)}{(2^{i+j-d+1}-1)(2^{2d-(i+j)+1}-1)} < \frac{2^{d+1}}{(2^2-1)(2^d-1)} < \frac{16}{21}.
$$
\item If $n = 2d+1$ and $d-j = i$, then
$\displaystyle{
\frac{|T_{h-1}|}{|T_h|} 
= \frac{(2^{d-j}-1)(2^{j+1}-1)(2^j-1)}{2^j(2^{d+1}-1)}}.
$ \qedhere
\end{enumerate}
\end{proof}
%----------------------------------------------

%------------------------1.7------------
\begin{lemma}\label{Lemma0.9}
Let $q=2$, $n\geq 2d+1$, $0\leq i \leq d-1$, $1\leq j \leq d$, $h := h_{\max}(i,j)$, and $h' := h_{\max}(i+1,j)$. Then
$\displaystyle{\frac{|T_h(i,j)|}{|T_{h'}(i+1,j)|} \geq l}$, where,
\begin{enumerate}[(i)]
\item $l = 3$, if $(n,h) \neq (2d+1,d-j)$ and $2\leq j\leq  d$;
\item $l = 3/2$, if $(n,h) \neq (2d+1,d-j)$ and $j = 1$;
\item $l = 6$, if $n = 2d+1$ and $d-j < i$;
\item $\displaystyle{l = \frac{2^{j}(2^{d+1}-1)}{(2^{j}-1)(2^{d-j+1}-1)}}$, if $n =2d+1$ and $d-j = i$.
\end{enumerate}
\end{lemma}
\begin{proof}
If $i < d-j$, then $d-i \geq 2$, $h = i$, and $h' = i+1$. So, 
$
\displaystyle{\frac{|T_h(i,j)|}{|T_{h'}(i+1,j)|} = \frac{2^{d-i}-1}{2^{d-i-j}-1} > \frac{3\cdot 2^j}{4}.}
$
If $d-j \leq i$, then $n-i-j\geq 2$, $h = d-j$, and $h' = d-j$. So, 
$$
\frac{|T_h(i,j)|}{|T_{h'}(i+1,j)|} 
= \frac{2^{d-i}(2^{i+j-d+1}-1)(2^{n-i-j}-1)}{(2^{d-i}-1)(2^{i+1}-1)}
> \frac{2^{d-i}\cdot 2^{i+j-d} \cdot 2^{n-i-j-1} \cdot 3/2}{2^{d-i} \cdot 2^{i+1}}
= 3\cdot 2^{n-d-i-3}.$$
Thus, 
\begin{inparaenum}[(i)]
\item $l = 3$, if $(n,h) \neq (2d+1,d-j)$ and $2\leq j\leq  d$; and
\item $l = 3/2$, if $(n,h) \neq (2d+1,d-j)$ and $j = 1$. 
    
\noindent \item If $n=2d+1$ and $d-j < i$, then $d+1 \leq i+j \leq 2d-1$. So,
$$
\frac{|T_h(i,j)|}{|T_{h'}(i+1,j)|} 
= \frac{2^{d-i}(2^{(i+j)-d+1}-1)(2^{2d-(i+j)+1}-1)}{(2^{d-i}-1)(2^{i+1}-1)}
> \frac{2(2^2-1)(2^d-1)}{(2-1)(2^d-1)} = 6.$$

\noindent \item If $n=2d+1$ and $d-j = i$, then
$\displaystyle{
\frac{|T_h(i,j)|}{|T_{h'}(i+1,j)|} 
= \frac{2^{j}(2^{d+1}-1)}{(2^{j}-1)(2^{d-j+1}-1)}}.
$
\end{inparaenum}
\end{proof}
%-----------------------------------------------------

%------------------------Theorem10--------------------
\begin{theorem}\label{q=2} Let $q = 2$ and $n\geq 2d+1$. If $0\leq i \leq d-1$ and $1\leq j \leq d$, then $|G_j(i+1)| < |G_j(i)|.$
\end{theorem}

\begin{proof} We use Lemmas \ref{BoundsGrassman}, \ref{Lemma0.8}, and \ref{Lemma0.9}. We have $\displaystyle{\frac{|G_j(i)|}{|G_j(i+1)|} > k\frac{|T_h(i,j)|}{|T_{h'}(i+1,j)|} > kl}$. For (i) - (iii), clearly $kl > 1$. For (iv), $\displaystyle{kl = \frac{(2^j-1)(2^{d-j+1}-1)+(2^{d-j}-1)(2^j+1)+2^{j+1}(2^j-1)+1}{(2^{j}-1)(2^{d-j+1}-1)}} >1$. Thus, $|G_j(i+1)| < |G_j(i)|$. 
\end{proof}

%----------------------------------------------------------------------
%-------------Bilinear-Forms-Graphs------------------------------------
%----------------------------------------------------------------------
\section{Proof of Conjecture \ref{Bilinear3.1}}

The terms corresponding to the indices $h$ such that $h > \min\{j,d-i\}$ in the above expression are zero. Therefore, we write the expression as 
$$
B_j(i) = \sum_{h=0}^{h_{\max}(i,j)} (-1)^{j-h}q^{eh+\binom{j-h}{2}} { d-h \brack d-j}{d-i \brack h}
$$
where $h_{\max}(i,j):= \min\{j,d-i\}$. Whenever $i,j$ are clear by the context we just write $h_{\max}$. Let us denote the $h$-th term of this expression by $T_h(i,j)$. Whenever $i,j$ are clear by the context we just write $T_h$.

Let $b_h(i,j)$ (if $i,j$ are clear by context we just write $b_h$) be the exponent of $q$ in $T_h(i,j)$ if we approximate ${n \brack k}$ with $q^{k(n-k)}$, that is,
$$
b_{h}(i,j) = h(d+e-i-h) + (d-j)(j-h)+ \binom{j-h}{2} = -\frac{1}{2}h^2+\left(e-i+\frac{1}{2}\right)h+j\left(d-\frac{(j+1)}{2}\right).
$$
Let $h_0 = e-i+\frac{1}{2}$. Then the quadratic expression $b_{h}(i,j)$ is maximal for $h = h_0$ and $b_{h_0+x} = b_{h_0} -\frac{1}{2}x^2$. The terms occurring in the sum have indices $h$ with $h \leq h_{\max} < h_0$, so the term with largest index has largest exponent.

%----------------------------------Lemma2-------------------------------
\begin{lemma}\label{Bilinear3.2}
If $q \geq 3$, or $q = 2$ and $d \neq e$, then $\displaystyle{|T_{h_{\max}}(i,j)|-|T_{h_{\max} - 1}(i,j)| \leq  |B_j(i)| \leq |T_{h_{\max}}(i,j)|}$, where $T_{-1}(i,j) = 0$. Moreover, the sign of  $B_j(i)$ is $(-1)^{\max(0,j+i-d)}$.
\end{lemma}
\begin{proof}
We note two inequalities. First, $b_h - b_{h-1} = e-i-h+1 \geq 2$ for all $1\leq h \leq h_{\max}$ except if $d=e$ and $h = d-i$. Second, $q^{b_h}< |T_h| < 2^2 q^{b_h}$ for all $q\geq 3$ by using Lemma \ref{lem1} (i) and (iii). These two inequalities imply that $\displaystyle{\frac{|T_h|}{|T_{h-1}|}> \frac{q^{b_h-b_{h-1}}}{4} \geq \frac{q^2}{4} > 1}$
for all $1\leq h \leq h_{\max}$ and $q\geq 3$ except if $d=e$ and $h = d-i$. However, if $q\geq 3$, $d=e$, and $h = d-i$, then
$$
\frac{|T_{d-i}|}{|T_{d-i-1}|} =\left|\frac{q^d(q^{j-(d-i)+1}-1)(q-1)}{q^{j-(d-i)(q^{i+1}-1)(q^{d-i}-1)}}\right| > (q-1)(1-q^{-(j-(d-i)+1)}) > (q-1)(1-q^{-1}) \geq \frac{4}{3} > 1. 
$$
On the other hand, if $q =2$ and $d\neq e$, then $e-i-h \geq 1$. Thus, for all $1\leq h \leq h_{\max}$ we have that
$$
\frac{|T_{h}|}{|T_{h-1}|}= \frac{(2^{j-h+1}-1)(2^{d-i-h+1}-1)}{2^{-e+j-h}(2^{d-h+1}-1)(2^h-1)}
> \frac{2^{j-h+d-i-h}}{2^{-e+j-h+d-h+1+h}}
= 2^{e-i-h-1}
\geq  1.
$$
Therefore, if $q \geq 3$, or $q=2$ and $d\neq e$, either the expression for $B_j(i)$ is of only one non-zero term or it is an alternating series with terms increasing is absolute value, from $h = 0$ to $h_{\max}$. Thus, by using Lemma \ref{mainlemma} we obtain both of the assertion because sign of the main term, ${T_{h_{\max}}(i,j)}$, is $(-1)^{\max(0,j+i-d)}$.
\end{proof}

%----------------------------------Lemma3--------------------------
\begin{lemma} \label{Bilinear3.3}
Let $q\geq 3$ and $s := b_{h_{\max}}(i,j)$. If $0 \leq i \leq d$ and $1\leq j \leq d$, then $\displaystyle{\frac{1}{4}q^{s} < |B_{j}(i)| < 2q^s}$. 
\end{lemma}

\begin{proof}
Let $h = h_{\text{max}}$. Then either ${d-h \brack d-j}$ or ${d-i \brack h}$ is equal to $1$. By using Lemma \ref{lem1} (iii) and Lemma \ref{Bilinear3.2} we obtain the upper bound $|B_j(i)| \leq |T_{h}| < 2q^s$. For the lower bound, if $h=0$, then we are done by Lemma \ref{lem1} (i). For $h \geq 1$, we consider two different cases. We use Lemma \ref{Bilinear3.2}, Lemma \ref{lem1} (ii), (iii), (iv), and the fact that $b_{h}-b_{h-1} \geq 1$. First, let $1\leq j \leq d-1$ and $j \neq d-i$. In this case there is exactly one Gaussian coefficient in $T_{h}$ that is not equal to 1, thus, 
$$
|B_{j}(i)| \geq |T_{h}| - |T_{h-1}|
> \left(1+\frac{1}{q}\right)q^s - 2\cdot \frac{q}{q-1}\cdot q^{s-1}
= q^s\left(1+\frac{1}{q} -\frac{2}{q-1}\right)
\geq  \frac{1}{3} q^s 
> \frac{1}{4} q^s.
$$
Second, let $j = d-i$ or $j=d$. In this case both Gaussian coefficients in $T_{h}$ are equal to $1$, thus,
\begin{align*}
|B_{j}(i)| \geq |T_{h}| - |T_{h-1}|
> q^s - \left(\frac{q}{q-1}\right)^2\cdot q^{s-1}
= q^s\left(1 -\frac{q}{(q-1)^2}\right)
\geq  \frac{1}{4} q^s. \tag*{\qedhere} 
\end{align*} 
\end{proof}

%-------------------------------Theorem4--------------------
\begin{theorem}\label{Bilinear3.4}
Let $q\geq 3$. If $0 \leq i \leq d-1$ and $1\leq j \leq d$, then $|B_j(i)| > |B_j(i+1)|$. 
\end{theorem}
\begin{proof}
Let $h := h_{\text{max}}(i,j)$, $h' := h_{\text{max}}(i+1,j)$, $s := b_{h}(i,j)$, and $s' := b_{h'}(i+1,j)$. If $d-i \leq j$, then $s-s'=e-(i+1)$ and if $j < d-i$, then $s-s'=j$. Hence, $s - s' \geq 2$ except if \begin{inparaenum}[(i)]
\item $e=d$ and $i =d-1$, or \item $j=1$. \end{inparaenum} Thus, Lemma \ref{Bilinear3.3} implies that
$\displaystyle{\frac{|B_j(i)|}{|B_j(i+1)|} > \frac{\frac{1}{4}\cdot q^{2}}{2}> 1}$
except if either (i) or (ii) holds. If either (i) or (ii) holds, we can find better bounds for $|B_j(i)|$ than given in Lemma \ref{Bilinear3.3}, with that we have $s-s'\geq 1$. That is, if either (i) or (ii) holds but not both, then 
${|B_{j}(i)| > 2/3 \cdot q^{s}}$. Thus, by again using Lemma \ref{Bilinear3.3} we obtain that $\displaystyle{\frac{|B_j(i)|}{|B_j(i+1)|} > \frac{\frac{2}{3}\cdot q^{1}}{2}\geq  1}$. Lastly, if both (i) and (ii) hold, then $|B_j(i)| > 1/2\cdot q^s$ and $|B_j(i+1)| < 3/2 \cdot q^{s'}$, and thus $|B_j(i)| >|B_j(i+1)|$.   
\end{proof}

%---------------------------Lemma5----------------
\begin{lemma}\label{Bilinear3.5}
Let $q = 2$, $0\leq i \leq d$, $1\leq j \leq d$, and $h:= h_{\max}(i,j)$. Then $\displaystyle{1 \geq \frac{|B_j(i)|}{|T_{h}(i,j)|} \geq k}$, where,
\begin{enumerate}[(i)]
\item $k = 1/3$, if $d-i < j$ and $e\geq d+1$;
\item $k = 2/3$, if $d-i > j$ and $e\geq d+1$;
\item $k = 1/2$, if $d-i =j$ and $e \geq d+2$;
\item $\displaystyle{k = \frac{2^{d+1-j}+2^j-1}{2^{d+1}}}$, if $d-i = j$ and $e = d+1$.
\end{enumerate}
\end{lemma}
\begin{proof} If $h =0$, then assertion follows since $\displaystyle{\frac{|B_j(i)|}{|T_h|} = 1}$. So, we assume that $h \geq 1$ and by Lemma \ref{Bilinear3.2} we have $\displaystyle{1 \geq \frac{|B_j(i)|}{|T_h|} \geq 1-\frac{|T_{h-1}|}{|T_h|}}$. Hence, we prove the assertion by finding the upper bounds for 
$$
\frac{|T_{h-1}|}{|T_h|} = \frac{2^{-e+j-h}(2^{d-h+1}-1)(2^{h}-1)}{(2^{j-h+1}-1)(2^{d-i-h+1}-1)}.
$$

\begin{enumerate}[(i)] 
\item If $d-i < j$, then $h = d-i$. So,
$$
\displaystyle{\frac{|T_{d-i-1}|}{|T_{d-i}|}} 
= \frac{2^{-e+j-(d-i)}(2^{i+1}-1)(2^{d-i}-1)}{(2^{j-(d-i)+1}-1)}
< \frac{2^{-e+j-(d-i)}2^{i+1}2^{d-i}}{2^{j-(d-i)}\cdot 3/2}
= \frac{2}{3 \cdot 2^{e-d-1}}
< \frac{2}{3}.
$$
\item If $d-i > j$, then $h = j$. So,
$$\displaystyle{\frac{|T_{j-1}|}{|T_{j}|}}
= \frac{2^{-e}(2^{d-j+1}-1)(2^j-1)}{(2^{(d-i)-j+1}-1)}
<  \frac{2^{-e}2^{d-j+1}2^{j}}{2^{(d-i)-j}\cdot 3/2}
< \frac{2}{3 \cdot 2^{e-1-i-j}}
< \frac{2}{3 \cdot 2^1} = \frac{1}{3}. 
$$

\item If $d-i = j$ and $e \geq d+2$, then $h = j$. So,
$$
\displaystyle{\frac{|T_{j-1}|}{|T_j|}} 
= 2^{-e}(2^{d-j+1}-1)(2^{j}-1)
<  2^{-e}2^{d-j+1}2^{j}
= \frac{2}{2^{e-d}}
< \frac{1}{2}.$$
\item If $d-i = j$ and $e = d+1$, then $h = j$. So,
\begin{align*}
\frac{|T_{j-1}|}{|T_j|}
= \frac{(2^{d-j+1}-1)(2^{j}-1)}{2^{d+1}} = 1- \frac{(2^{d+1-j}+2^j-1)}{2^{d+1}}. \tag*{\qedhere} 
\end{align*}
\end{enumerate} 
\end{proof}

%----------------------Lemma6----------------
\begin{lemma}\label{Bilinear3.6}
Let $q=2$, $0\leq i \leq d-1$, $1\leq j \leq d$, $h := h_{\max}(i,j)$ and $h' := h_{\max}(i+1,j)$. Then
$\displaystyle{\frac{|T_h(i,j)|}{|T_{h'}(i+1,j)|} > l}$, where,  
\begin{enumerate}[(i)]
\item $l = 3$, if $d-i < j$ and $e\geq d+1$; 
\item $l = 3/2$, if $d-i >j$ and $e\geq d+1$; 
\item $l = 2$, if $d-i = j$ and $e\geq d+2$;
\item $l = 2^j$, if $d-i = j$ and $e = d+1$.
\end{enumerate}
\end{lemma}
\begin{proof}
\begin{enumerate}[(i)]
\item If $d-i < j$ and $e\geq d+1$, then $h = d-i$ and $h' = d-i-1$. So,
$$
\frac{|T_{d-i}(i,j)|}{|T_{d-i-1}(i+1,j)|}
= \frac{(2^{j -(d-i)+1}-1)}{2^{-e + j-(d-i)} (2^{i+1}-1)}
> \displaystyle{\frac{2^{j-(d-i)}\cdot 3/2}{2^{-e+j-(d-i)} 2^{i+1}}}
= 3\cdot 2^{e-i-2}
\geq 3.
$$

\item If $d-i >j$ and $e\geq d+1$, then $h =j$ and $h'=j$. So,
$$\frac{|T_{j}(i,j)|}{|T_{j}(i+1,j)|}
= \frac{(2^{d-i}-1)}{(2^{(d-i)-j}-1)}
> \frac{2^{d-i-1}\cdot 3/2}{2^{d-i-j}}
= \frac{3}{2}\cdot 2^{j-1}
\geq \frac{3}{2}.
$$
\item If $d-i = j$ and $e\geq d+2$, then $h = j$ and $h'=j-1$. So,
$$\frac{|T_{j}(i,j)|}{|T_{j-1}(i+1,j)|}
= \frac{2^e}{(2^{i+1}-1)}
> \frac{2^e}{2^{i+1}}
= 2^{e-i-1}
\geq 2.$$
\item If $d-i = j$ and $e = d+1$, then $h = j$ and $h'=j-1$. So,
$$\frac{|T_{j}(i,j)|}{|T_{j-1}(i+1,j)|}
= \frac{2^{d+1}}{(2^{i+1}-1)}
> \frac{2^{d+1}}{2^{i+1}}
= 2^{j}.$$
\end{enumerate}
This shows the assertion. 
\end{proof}

%-------------------------------Theorem7-----------------------
\begin{theorem}\label{Bilinear3.7}
Let $q=2$ and $e\geq d+1$. If $0\leq i \leq d-1$ and $1\leq j \leq d$, then $|B_j(i)| > |B_j(i+1)|$.
\end{theorem}

\begin{proof}
We use Lemmas \ref{Bilinear3.5} and \ref{Bilinear3.6}. We have $\displaystyle{\frac{|B_j(i)|}{|B_j(i+1)|}} \geq k\frac{|T_h(i,j)|}{|T_{h'}(i+1,j)|} > kl$. For (i) - (iii), $kl = 1$. For (iv), $\displaystyle{kl = \frac{2^{d+1}+2^{2j}-2^j}{2^{d+1}}> 1}$. Thus, $|B_j(i)| > |B_j(i+1)|$.
\end{proof}

%--------------------------Theorem8--------------------------
\begin{theorem} \label{Bilinear3.8}
For $q \geq 3$, or $q = 2$ and $d \neq e$, $B_j(d - j + 1)$ is the smallest eigenvalue in the distance-$j$ graph for $1 \leq  j \leq d$. 
\end{theorem}
\begin{proof}
By Theorem \ref{Bilinear3.4} and \ref{Bilinear3.7} we just need to find the smallest value of $i$ for which $B_j(i)$ is negative. By Lemma \ref{Bilinear3.2} the negative terms are $B_j(d-j + 1 + 2t)$ for $t\geq 0$. Hence, $B_j(d - j + 1)$ is the smallest eigenvalue for a fixed $j$. 
\end{proof}

%----------------------------------------------------------------------
%-------------------------------Hermitian Forms Graphs-----------------
%----------------------------------------------------------------------
\section{Proofs of Conjecture \ref{conj1} and Conjecture \ref{conj2}}

 The term corresponds to an index $h$ such that $h > \min\{j,d-i\}$ in the above expression is zero. Therefore, we write the expression as 
\begin{align}\label{grassmanexpression}
Q_j(i) = \sum_{h=0}^{h_{\max}(i,j)}(-1)^j (-q)^{\binom{j-h}{2}+hd}{d-h \brack d-j}_b{d-i \brack h}_b
\end{align}
where $h_{\max}(i,j):= \min\{j,d-i\}$. Whenever $i,j$ are clear by the context we just write $h_{\max}$. Let us denote the $h$-th term of sum in this expression by $T_h(i,j)$. Whenever $i,j$ are clear by the context we just write $T_h$.

Let $q_{h}(i,j)$ (if $i,j$ are clear by context we just write $q_h$) be the exponent of $q$ in $|T_h(i,j)|$ if we approximate $|{n \brack k}_{-q}|$ with $\displaystyle{q^{k(n-k)}}$, that is,
\begin{align*}
q_{h}(i,j) &= (d-j)(j-h)+h(d-i-h)+\frac{(j-h)(j-h-1)}{2} +hd\\
&= -\frac{1}{2}h^2 + \left(d-i+\frac{1}{2}\right)h + j\left(d-\frac{(j+1)}{2}\right)
\end{align*}
Let $h_0 = d-i+\frac{1}{2}$. Then the quadratic expression $q_{h}$ is maximal for $h = h_0$, and $q_{h_0+x} = q_{h_0}-\frac{1}{2}x^2$.  The terms occurring in the sum have indices $h$ with $h \leq h_{\max} < h_0$, so the term with the largest index has the largest exponent. Note that, if $m$ is even, then $(1-q^{-m})q^m \leq |b^m -1| \leq q^m$. On the other hand, if $m$ is odd, then $q^m \leq |b^m-1| \leq  (1+q^{-m})q^m$. Thus, we obtain the following useful estimate (for $m >0$),
\begin{align}\label{estimate}
(1-q^{-m})q^m &\leq |b^m -1| \leq (1+q^{-m})q^m. 
\end{align}
We will use the following ratio of absolute values of consecutive terms,
\begin{align}\label{consecutive-ratio}
    \left|\frac{T_{h-1}(i,j)}{T_h(i,j)}\right| = \left|\frac{b^{j-h-d}(b^{d-h+1}-1)(b^h-1)}{(b^{j-h+1}-1)(b^{d-i-h+1}-1)} \right|
\end{align}
for any $h = 1,...,h_{\max}(i,j)$. Also, we will need that 
\begin{align}\label{useful-ratio}
    \left|\frac{T_{d-i-2}(i,j)}{T_{d-i}(i,j)}\right| = \left|\frac{b^{2(j-(d-i))+1-2d}(b^{i+2}-1)(b^{i+1}-1)(b^{d-i}-1)(b^{d-i-1}-1)}{(b^2-1)(b-1)(b^{j-(d-i)+2}-1)(b^{j-(d-i)+1}-1)} \right|
\end{align}
whenever $2\leq d-i \leq j$.

%--------------------------------------------
\begin{lemma}\label{lemma3}
Let $q\geq 2$ and $d\geq 6$. If $j=1$ and $1\leq i \leq d-1$, then $|Q_j(i)| > |Q_j(i+1)|$. Moreover, the sign of $Q_j(i)$ is same as that of $T_{h_{\max}}(i,j)$.
\end{lemma}
\begin{proof} For all $1\leq i \leq d$, we have $\displaystyle{Q_1(i)= \frac{
b^{2d-i}-1}{b-1}}$. Its sign is $(-1)^{i+1}$, same as that of $T_{h_{\max}}(i,j)$. Now, we use the estimate from (\ref{estimate}) to conclude that
\begin{align*}
\frac{|Q_1(i)|}{|Q_1(i+1)|} = \left|\frac{b^{2d-i}-1}{b^{2d-i-1}-1}\right|
\geq \frac{(1-q^{-(2d-i)})q}{(1+q^{-(2d-i-1)})} \geq \frac{(1-q^{-(d+1)})q}{(1+q^{-d})} \geq \frac{(1-q^{-7}) q}{1+q^{-6}} \geq\frac{127}{65} > 1. \mbox{\qedhere}
\end{align*}
\end{proof}
%-------------------------------------------

%-------------------------------------------
\begin{lemma}\label{terms_inc}
Let $q\geq 2$ and $d\geq 6$. If $j\geq 2$ and $1\leq i\leq d-3$, then the terms of $Q_j(i)$ increases in absolute value, that is, $|T_{h-1}(i,j)| < |T_h(i,j)|$ from $h = 1,...,h_{\max}(i,j)$.  
\end{lemma}
\begin{proof}
We divide the proof into three parts depending on the values of $j$. First, suppose that $j \leq d-i-1$, so $h_{\max} = j$. By using (\ref{consecutive-ratio}) together with the estimate (\ref{estimate}) we obtain that
\begin{align}\label{eqn1}
    \left|\frac{T_{h-1}(i,j)}{T_{h}(i,j)}\right| \leq  \frac{(1+q^{-(d-h+1)})(1+q^{-h})q^{-(d-i-h+1)}}{(1-q^{-(j-h+1)})(1-q^{-(d-i-h+1)})}
\leq  \frac{(1+q^{-4})(1+q^{-1})q^{-3}}{(1-q^{-2})(1-q^{-3})} \leq \frac{17}{56} < 1
\end{align}
for all $1\leq h \leq j-1$, and $h =j$, that
\begin{align}\label{eqn2}
    \left|\frac{T_{j-1}(i,j)}{T_j(i,j)}\right| \leq  \frac{(1+q^{-(d-j+1)})(1+q^{-j})q^{-(d-i-j)}}{(1+q)(1-q^{-(d-i-j+1)})}
\leq  \frac{(1+q^{-3})(1+q^{-2})q^{-1}}{(1+q)(1-q^{-2})} \leq  \frac{5}{16} < 1.
\end{align}
Second, suppose that $j = d-i$, so $h_{\max} = j$. We use (\ref{consecutive-ratio}) together with the estimate (\ref{estimate}) to obtain that
\begin{align}\label{eqn3}
    \left|\frac{T_{h-1}(i,j)}{T_{h}(i,j)}\right| \leq  \frac{(1+q^{-(d-h+1)})(1+q^{-h})q^{-(j-h+1)}}{(1-q^{-(j-h+1)})(1-q^{-(j-h+1)})} \leq  \frac{(1+q^{-3})(1+q^{-1})q^{-2}}{(1-q^{-2})(1-q^{-2})} \leq  \frac{3}{4} < 1
\end{align}
for all $1\leq h \leq j-1$, and $h =j = d-i$, that
\begin{align}\label{eqn4}
    \left|\frac{T_{d-i-1}(i,j)}{T_{d-i}(i,j)}\right| \leq \frac{(1+q^{-(d-j+1)})(1+q^{-j})q}{(1+q)^2} \leq  \frac{(1+q^{-2})(1+q^{-3})q}{(1+q)^2} \leq  \frac{5}{16} <1.
\end{align}
Lastly, suppose that $j \geq d-i+1$, so $h_{\max} = d-i$. Again using (\ref{consecutive-ratio}) together with the estimate (\ref{estimate}), one gets that
\begin{align}\label{eqn5}
    \left|\frac{T_{h-1}(i,j)}{T_{h}(i,j)}\right| \leq  \frac{(1+q^{-(d-h+1)})(1+q^{-h})q^{-(d-i-h+1)}}{(1-q^{-(j-h+1)})(1-q^{-(d-i-h+1)})}
\leq  \frac{(1+q^{-3})(1+q^{-1})q^{-2}}{(1-q^{-3})(1-q^{-2})} \leq  \frac{9}{14} < 1
\end{align}
for all $1\leq h \leq d-i-1$, and $h =d-i$, that
\begin{align}\label{eqn6}
    \left|\frac{T_{d-i-1}(i,j)}{T_{d-i}(i,j)}\right| \leq  \frac{(1+q^{-(i+1)})(1+q^{-(d-i)})}{(1-q^{-(j-(d-i)+1)})(1+q)} \leq
    \frac{(1+q^{-2})(1+q^{-3})}{(1-q^{-2})(1+q)} \leq \frac{5}{8} < 1.
\end{align}
This show the assertion.
\end{proof}
%-------------------------------------------

%--------------------------------------------------------------------------
\begin{lemma}
Let $q\geq 2$ and $d\geq 6$. If $j \geq 2$ and $1\leq i \leq d-3$, then 
\begin{align}\label{upperbound}
|Q_j(i) - T_{h_{\max}}(i,j)| \leq \begin{cases}
|T_0| + |T_1| + ... + |T_{j-1}|, &\text{if } j \leq d-i-1\\
|T_{d-i-2}| + |T_{d-i-1}|, &\text{if } j\geq d-i.
\end{cases}
\end{align}
\end{lemma}
\begin{proof} If $j \leq d-i-1$, then $h_{\max} = j$. Thus, the first inequality in (\ref{upperbound}) follows by the triangle's inequality. If $h_{\max} = d-i$ and is odd, then signs of $T_h$ for $0\leq h \leq h_{\max}$ follow the pattern 
\begin{align*}
    &+ |T_0|, -|T_1|, - |T_2|, + |T_3|, + |T_4|, -|T_5|,  -|T_6|, ..., T_{d-i}, \text{ or}\\
    &- |T_0|, +|T_1|, + |T_2|, - |T_3|, - |T_4|, +|T_5|,  +|T_6|, ..., T_{d-i}.
\end{align*}
On the other hand, if $h_{\max} = d-i$ and is even, then signs of $T_h$ for $0\leq h \leq h_{\max}$ follow the pattern 
\begin{align*}
    &+ |T_0|, +|T_1|, - |T_2|, -|T_3|, +|T_4|, +|T_5|, - |T_6|,..., T_{d-i}, \text{ or}\\
    &- |T_0|, -|T_1|, + |T_2|, + |T_3|, - |T_4|, - |T_5|, + |T_6|,..., T_{d-i}.
\end{align*}
That means if $j \geq d-i$, then sign of $T_{d-i}$ is opposite to that of $T_{d-i-1}$ and $T_{d-i-2}$. Thus, the second inequality in (\ref{upperbound}) follows by the Lemma \ref{terms_inc}, Lemma \ref{mainlemma}, and the triangle's inequality. \end{proof}
%----------------------------------------

%----------------------------------------
\begin{lemma}\label{threeparts}
Let $q\geq 2$, $d\geq 6$, $j\geq 2$, and $1\leq i\leq d-3$.
\begin{enumerate}[(i)]
   \item If $j\leq d-i-1$, then $\displaystyle{43/78\cdot|T_{h_{\max}}(i,j)| \leq |Q_j(i)| \leq 113/78\cdot |T_{h_{\max}}(i,j)|}$.
   \item If $j= d-i$, then $\displaystyle{691/1296\cdot |T_{h_{\max}}(i,j)| \leq |Q_j(i)| \leq 1901/1296 \cdot|T_{h_{\max}}(i,j)|}$.
   \item If $j \geq d-i+1$, then $\displaystyle{31/216\cdot|T_{h_{\max}}(i,j)| \leq |Q_j(i)| \leq 401/216\cdot |T_{h_{\max}}(i,j)|}$.
\end{enumerate}
In particular, the sign of $Q_j(i)$ is same as that of $T_{h_{\max}}(i,j)$. 
\end{lemma}
\begin{proof}
\begin{enumerate}[(i)]
\item If $j \leq d-i-1$, then $h_{\max} =j$. We use (\ref{upperbound}) along with inequalities in (\ref{eqn1}) and (\ref{eqn2}) to get that
    \begin{align*}
        |Q_j(i) - T_j(i,j)| <  \left(\sum_{h\geq 0} \left(\frac{17}{56} \right)^h \right)|T_{j-1}(i,j)| \leq \left(\sum_{h\geq 0} \left(\frac{17}{56} \right)^h \right)\frac{5}{16}|T_j(i,j)| = \frac{35}{78}|T_j(i,j)|.
    \end{align*}
Thus, we obtain that $\displaystyle{43/78\cdot|T_{j}(i,j)| \leq |Q_j(i)| \leq 113/78\cdot |T_j(i,j)|}$.
\item If $j = d-i$, then $h_{\max}=d-i$. By using (\ref{useful-ratio}) together with the estimates before (\ref{estimate}) we obtain that
\begin{align*}
    \left|\frac{T_{d-i-2}(i,j)}{T_{d-i}(i,j)}\right| \leq  \frac{q^{i+2}(1+q^{-(i+1)})q^{d-i}(1+q^{-(d-i-1)})q^d}{q^{2d-1}(q^2-1)^2(q+1)^2} \leq \frac{(1+q^{-2})(1+q^{-2})q^3}{(q^2-1)^2(q+1)^2} \leq  \frac{25}{162}.
\end{align*}
This inequality together with (\ref{eqn4}) in second equation of (\ref{upperbound}) implies that
\begin{align*}
|Q_j(i) - T_{d-i}(i,j)|  \leq  \frac{25}{162} |T_{d-i}(i,j)| + \frac{5}{16}|T_{d-i}(i,j)| = \frac{605}{1296}|T_{d-i}(i,j)|.
\end{align*}
Thus, we conclude that $\displaystyle{691/1296\cdot |T_{d-i}(i,j)| \leq |Q_j(i)| \leq 1901/1296 \cdot|T_{d-i}(i,j)|}$.
\item If $j \geq d-i+1$, then $h_{\max}=d-i$. By using (\ref{useful-ratio}) together with the estimates before (\ref{estimate}) we obtain that
\begin{align*}
\left|\frac{T_{d-i-2}(i,j)}{T_{d-i}(i,j)}\right| \leq
\frac{(1+q^{-(i+1)})(1+q^{-(d-i-1)})}{(1-q^{-(j-(d-i)+1)})(1+q)(q^2-1)}\leq \frac{(1+q^{-2})(1+q^{-2})}{(1-q^{-2})(1+q)(q^2-1)} \leq  \frac{25}{108}.
\end{align*} 
This inequality together (\ref{eqn6}) in second equation of (\ref{upperbound}) implies that
\begin{align*}
|Q_j(i) - T_{d-i}(i,j)| \leq \frac{25}{108} |T_{d-i}(i,j)| + \frac{5}{8}|T_{d-i}(i,j)| = \frac{185}{216}|T_{d-i}(i,j)|.
\end{align*}
Thus, we conclude that $\displaystyle{31/216\cdot|T_{d-i}(i,j)| \leq |Q_j(i)| \leq 401/216\cdot |T_{d-i}(i,j)|}$. \qedhere
\end{enumerate} 
\end{proof}
%-------------------------------------------------------------------------

%--------------------------------------------
\begin{lemma}\label{lemma7}
Let $q\geq 2$ and $d\geq 6$. If $1\leq i \leq d-5$ and $j\geq 2$, or $i =d-4$ and $2\leq j \leq 4$ then $|Q_j(i)| > |Q_j(i+1)|$.
\end{lemma}
\begin{proof}
Let $h := h_{\max}(i,j)$ and $h' := h_{\max}(i+1,j)$. We divide the proof into four parts depending on the values of $j$. First, if $j \leq d-i-2$, then $h = j$ and $h'=j$. By using the estimate from (\ref{estimate}) we obtain that 
\begin{align*}
\frac{|T_h(i,j)|}{|T_{h'}(i+1,j)|} =  \left|\frac{b^{d-i}-1}{b^{d-i-j} -1}\right| \geq \frac{(1-q^{-(d-i)})q^j}{1+q^{-(d-i-j)}} \geq \frac{(1-q^{-4})q^2}{1+q^{-2}} \geq 3.
\end{align*}
This inequality together with Lemma \ref{threeparts} (i) imply that 
$$\frac{|Q_j(i)|}{|Q_j(i+1)|} \geq \frac{43}{113}\frac{|T_h(i,j)|}{|T_{h'}(i+1,j)|} \geq \frac{43}{113}\cdot 3 >1.$$
Second, if $j = d-i-1$, then $h=j$ and $h' = j$. By using the estimate from (\ref{estimate}) we obtain that 
\begin{align*}
\frac{|T_h(i,j)|}{|T_{h'}(i+1,j)|} = \left|\frac{b^{d-i}-1}{b^{d-i-j} -1}\right| \geq \frac{(1-q^{-(d-i)})q^j}{1+q^{-(d-i-j)}} \geq \frac{(1-q^{-4})q^3}{1+q^{-1}} \geq 5.
\end{align*}
This inequality together with Lemma \ref{threeparts} (i) and (ii) imply that 
$$\frac{|Q_j(i)|}{|Q_j(i+1)|} \geq \frac{43/78}{1901/1296}\frac{|T_h(i,j)|}{|T_{h'}(i+1,j)|} \geq \frac{43/78}{1901/1296}\cdot 5 > 1.$$
Third, if $j = d-i$, then $h=d-i$ and $h' = d-i-1$. By using the estimate from (\ref{estimate}) we obtain that 
\begin{align*}
\frac{|T_{h}(i,j)|}{|T_{h'}(i+1,j)|} = \left|\frac{(b-1)b^{d}}{b^{i+1} -1}\right| \geq \frac{(1+q)q^{d-i-1}}{1+q^{-(i+1)}} \geq \frac{(1+q)q^3}{1+q^{-2}} \geq \frac{96}{5}.
\end{align*}
This inequality together with Lemma \ref{threeparts} (ii) and (iii) imply that 
$$\frac{|Q_j(i)|}{|Q_j(i+1)|} \geq \frac{691/1296}{401/216}\frac{|T_{h}(i,j)|}{|T_{h'}(i+1,j)|} \geq \frac{691/1296}{401/216}\cdot \frac{96}{5} > 1.$$
Lastly, if $j \geq d-i+1$, then $h=d-i$ and $h' = d-i-1$. By using the estimate from (\ref{estimate}) we obtain that 
\begin{align*}
\frac{|T_{h}(i,j)|}{|T_{h'}(i+1,j)|} = \left|\frac{(b^{i+j-d+1}-1)b^{-i-j+2d}}{b^{i+1} -1}\right|
 \geq \frac{(1-q^{-(j-(d-i)+1)})q^{d-i}}{1+q^{-(i+1)}} \geq \frac{(1-q^{-2})q^5}{1+q^{-2}} \geq \frac{96}{5}.
\end{align*}
This inequality together with Lemma \ref{threeparts} (iii) imply that 
\begin{align*}
\frac{|Q_j(i)|}{|Q_j(i+1)|} \geq \frac{31}{401}\frac{|T_{h}(i,j)|}{|T_{h'}(i+1,j)|} \geq \frac{31}{401}\cdot\frac{96}{5} > 1. \tag*{\qedhere} 
\end{align*}
\end{proof}
%--------------------------------------------

%--------------------------------------------
\begin{lemma}\label{lemma8}
Let $q\geq 2$ and $d\geq 6$. If $j \geq 2$ and $i = d-1$, then $|Q_j(i)| > |Q_j(i+1)| $ except if $q=2$ and $j$ is equal to $d$ or is even. However, if $q=2$ and $j$ is equal to $d$ or is even, then $|Q_j(d-2)| > |Q_j(d)|$. 
\end{lemma}
\begin{proof} We use the estimate from (\ref{estimate}). If $q\geq 3$, then
$$
\displaystyle{\left|\frac{b^d(b^j-1)}{b^{j-1}(b^d-1)} \right| \geq \frac{(1-q^{-j})q}{(1+q^{-d})} \geq \frac{(1-q^{-2})q}{(1+q^{-6})} \geq \frac{972}{365} > 2}.
$$
If $q= 2$, $j \leq d-1$, and $j$ is odd, then
$$
\displaystyle{\left|\frac{b^d(b^j-1)}{b^{j-1}(b^d-1)} \right| \geq \frac{(1+q^{-j})q}{(1+q^{-d})} \geq \frac{(1+q^{-(d-1)})q}{(1+q^{-d})} > 2}.
$$
Since sign of $\displaystyle{\frac{b^d(b^j-1)}{b^{j-1}(b^d-1)}}$ is $-1$, so 
$\displaystyle{\frac{|Q_j(d-1)|}{|Q_j(d)|} = \left|\frac{b^d(b^j-1)}{b^{j-1}(b^d-1)} + 1\right| \geq \left|\frac{b^d(b^j-1)}{b^{j-1}(b^d-1)}\right| - 1 > 1}$
except if $q=2$ and $j$ is equal to $d$ or is even. 

If $q= 2$ and $j$ is equal to $d$ or is even, then
$$
\displaystyle{\left|\frac{b^d(b^j-1)}{b^{j-1}(b^d-1)} \right| \geq \frac{(1-q^{-j})q}{(1+q^{-d})} \geq \frac{96}{65}}, \text{ and }\displaystyle{\left|\frac{b^d(b^{j-1}-1)}{b^{j-2}(b^{d-1}-1)} \right| \geq \frac{(1+q^{-(j-1)})q^2}{(1+q^{-(d-1)})} \geq q^2 = 4.}
$$
Thus, we obtain 
\begin{align*}
\frac{|Q_j(d-2)|}{|Q_j(d)|} = \left|1 + \frac{b^d(b^j-1)}{b^{j-1}(b^d-1)}\left( \frac{b^d(b^{j-1}-1)}{b^{j-2}(b^{d-1}-1)}-1\right)\right| \geq \frac{96}{65}\cdot (4-1) - 1 > 1. \tag*{\qedhere} 
\end{align*}
\end{proof}
%---------------------------------------------

%---------------------------------------------
\begin{lemma}\label{lemma9}
Let $q\geq 2$ and $d\geq 6$. If $j \geq 2$ and $i = d-2$, then $|Q_j(i)| > |Q_j(i+1)|$. Moreover, if $i = d-2,\ d-1$, then the sign of $Q_j(i)$ is same as that of $T_{h_{\max}}(i,j)$.
\end{lemma}
\begin{proof} Since $h_{\max}(i,j) = 2$ and $h_{\max}(i+1,j) = 1$, thus we have $\left|Q_j(i) - T_2(i,j)\right| = |T_0(i,j) + T_1(i,j)| \leq |T_0(i,j)| + |T_1(i,j)|$ and  $\left|Q_j(i+1) - T_1(i+1,j)\right| = |T_0(i+1,j)|$. We break the proof into three parts depending on the value of $j$. First, suppose that $j=2$. We use (\ref{consecutive-ratio}) and (\ref{useful-ratio}) together with the estimate from (\ref{estimate}) to obtain the following inequalities: $\displaystyle{\frac{|T_0(i,j)|}{|T_2(i,j)|} \leq \frac{(1+q^{-(d-1)})}{(q^2-1)(q+1)} \leq \frac{11}{96}}$, 
$\displaystyle{\frac{|T_1(i,j)|}{|T_2(i,j)|} \leq \frac{(1+q^{-(d-1)})(q^2-1)}{(1+q)^2q} \leq \frac{11}{64}}$, and
$\displaystyle{\frac{|T_0(i+1,j)|}{|T_1(i+1,j)|} \leq \frac{(1+q^{-d})q}{(q^2-1)}  \leq \frac{65}{96}}$.
Thus, by using the above three inequalities we obtain that
$$
\frac{137}{192}|T_2(i,j)| \leq |Q_j(i)| \leq \frac{247}{192}|T_2(i,j)| \text{ and } \frac{31}{96}|T_1(i+1,j)| \leq |Q_j(i+1)| \leq \frac{161}{96}|T_1(i+1,j)|.
$$
This shows the second assertion for $j=2$. By using the estimate from (\ref{estimate}) again we obtain that
$\displaystyle{\frac{|T_2(i,j)|}{|T_1(i+1,j)|} = \left|\frac{(b-1)b^d}{b^{d-1}-1} \right| \geq \frac{(1+q)q}{(1+q^{-(d-1)})} \geq \frac{64}{11}}.$
So,
$$
\frac{|Q_j(i)|}{|Q_j(i+1)|} \geq \frac{137/192|T_2(i,j)|}{161/96|T_1(i+1,j)|} \geq \frac{137/192}{161/96}\cdot \frac{64}{11} = \frac{4384}{1771} >1. 
$$
Second, suppose that $j=3$. We use (\ref{consecutive-ratio}) and (\ref{useful-ratio}) together with the estimate from (\ref{estimate}) to obtain the following inequalities: $\displaystyle{\frac{|T_0(i,j)|}{|T_2(i,j)|} \leq \frac{(1+q^{-(d-1)})q^2}{(q^3+1)(q^2-1)} \leq \frac{11}{72}}$, 
$\displaystyle{\frac{|T_1(i,j)|}{|T_2(i,j)|} \leq \frac{(1+q^{-(d-1)})}{(1+q)} \leq \frac{11}{32}}$, and
$\displaystyle{\frac{|T_0(i+1,j)|}{|T_1(i+1,j)|} \leq \frac{(1+q^{-d})q^2}{(1+q^3)}  \leq \frac{65}{144}}$.
Thus, by using the above three inequalities we obtain that
$$
\frac{145}{288}|T_2(i,j)| \leq |Q_j(i)| \leq \frac{431}{288}|T_2(i,j)| \text{ and } \frac{79}{144}|T_1(i+1,j)| \leq |Q_j(i+1)| \leq \frac{209}{144}|T_1(i+1,j)|.
$$
This shows the second assertion for $j=3$. By using the estimate from (\ref{estimate}) again we obtain that
$\displaystyle{\frac{|T_2(i,j)|}{|T_1(i+1,j)|} = \left|\frac{(b^2-1)b^{d-1}}{(b^{d-1}-1)} \right| \geq \frac{(q^2-1)}{(1+q^{-(d-1)})} \geq \frac{32}{11}}.$
So,
$$
\frac{|Q_j(i)|}{|Q_j(i+1)|} \geq \frac{145/288|T_2(i,j)|}{209/144|T_1(i+1,j)|} \geq \frac{145/288}{209/144}\cdot \frac{32}{11} = \frac{2320}{2299} >1. 
$$
Third, suppose that $j \geq 4$. Use (\ref{consecutive-ratio}) and (\ref{useful-ratio}) together with the estimate from (\ref{estimate}) to obtain the following inequalities: $\displaystyle{\frac{|T_0(i,j)|}{|T_2(i,j)|} \leq \frac{(1+q^{-(d-1)})}{(1-q^{-(j-1)})q^3} \leq \frac{33}{224}}$, 
$\displaystyle{\frac{|T_1(i,j)|}{|T_2(i,j)|} \leq \frac{(1+q^{-(d-1)})(q^2-1)}{(1-q^{-(j-1)})(1+q)q^2} \leq \frac{33}{112}}$, and
$\displaystyle{\frac{|T_0(i+1,j)|}{|T_1(i+1,j)|} \leq \frac{(1+q^{-d})}{(1-q^{-j})q}  \leq \frac{13}{24}}$.
Thus, by using the above three inequalities we obtain that
$$
\frac{125}{224}|T_2(i,j)| \leq |Q_j(i)| \leq \frac{323}{224}|T_2(i,j)| \text{ and } \frac{11}{24}|T_1(i+1,j)| \leq |Q_j(i+1)| \leq \frac{37}{24}|T_1(i+1,j)|.
$$
This shows the second assertion for $j\geq 4$. By using the estimate from (\ref{estimate}) again we obtain that
$\displaystyle{\frac{|T_2(i,j)|}{|T_1(i+1,j)|} = \left|\frac{b^d(b^{j-1}-1)}{b^{j-2}(b^{d-1}-1)} \right| \geq \frac{q^2(1-q^{-(j-1)})}{(1+q^{-(d-1)})} \geq \frac{112}{33}}.$
So,
\begin{align*}
\frac{|Q_j(i)|}{|Q_j(i+1)|} \geq \frac{125/224|T_2(i,j)|}{37/24|T_1(i+1,j)|} \geq \frac{125/224}{37/24}\cdot \frac{112}{33} = \frac{500}{407} >1. \tag*{\qedhere} 
\end{align*}
\end{proof}
%------------------------------------------

%------------------------------------------
\begin{lemma}\label{lemma10}
Let $q \geq 2$ and $d\geq 6$. If $j\geq 2$ and $i = d-3$, then $ |Q_j(i)| > |Q_j(i+1)|$. 
\end{lemma}
\begin{proof} We break the proof into three parts depending on the value of $j$. First, suppose that $j=2$, so $h_{\max}(i,j) = 2$ and $h_{\max}(i+1,j) =2$. We use (\ref{grassmanexpression}), (\ref{consecutive-ratio}) and (\ref{useful-ratio}) together with the estimate (\ref{estimate}) to obtain the following inequalities: $\displaystyle{\frac{|T_0(i,j)|}{|T_2(i,j)|} \leq \frac{(1+q^{-(d-1)})}{(q^3+1)(q^2-1)} \leq \frac{11}{288}}$, $\displaystyle{\frac{|T_1(i,j)|}{|T_2(i,j)|} \leq \frac{(1+q^{-(d-1)})}{(1+q)q} \leq \frac{11}{64}}$, $\displaystyle{\frac{|T_0(i+1,j)|}{|T_2(i+1,j)|} \leq \frac{(1+q^{-(d-1)})}{(q^2-1)(1+q)} \leq \frac{11}{96}}$, and $\displaystyle{\frac{|T_1(i+1,j)|}{|T_2(i+1,j)|} \leq \frac{(1+q^{-(d-1)})(q^2-1)}{(1+q)(1+q)q} \leq \frac{11}{64}}$. By using these inequalities we obtain that
$$
\frac{455}{576}|T_2(i,j)| \leq |Q_j(i)| \leq \frac{697}{576}|T_2(i,j)| \text{ and } \frac{137}{192}|T_2(i+1,j)| \leq |Q_j(i+1)| \leq \frac{247}{192}|T_2(i+1,j)|.
$$
Now, by using (\ref{estimate}) we obtain that
$\displaystyle{\frac{|T_2(i,j)|}{|T_2(i+1,j)|} = \left|\frac{b^3-1}{b-1} \right| \geq \frac{1+q^3}{1+q} \geq 3}.$
So
$$
\frac{|Q_j(i)|}{|Q_j(i+1)|} \geq \frac{455/576|T_2(i,j)|}{247/192|T_2(i+1,j)|} \geq \frac{455/576}{247/192}\cdot 3 = \frac{35}{19} >1. 
$$
Second, suppose that $j=3$. We use (\ref{consecutive-ratio}) and (\ref{useful-ratio}) together with the estimate (\ref{estimate}) to obtain the following inequalities: $\displaystyle{\frac{|T_1(i,j)|}{|T_3(i,j)|} \leq \frac{(1+q^{-(d-2)})(1+q^3)}{(q^2-1)(1+q)^2q^2} \leq \frac{17}{192}}$,\  \ $\displaystyle{\frac{|T_2(i,j)|}{|T_3(i,j)|} \leq \frac{(1+
q^{-(d-2)})(1+q^3)}{(1+q)^2q^2} \leq \frac{17}{64}}$, $\displaystyle{\frac{|T_0(i+1,j)|}{|T_2(i+1,j)|} \leq \frac{(1+q^{-(d-1)})q^{2}}{(q^3+1)(q^2-1)} \leq \frac{11}{72}}$, and
$\displaystyle{\frac{|T_1(i+1,j)|}{|T_2(i+1,j)|}\leq \frac{1+q^{-(d-1)}}{1+q} \leq \frac{11}{32}}$. By using these inequalities we obtain that
$$
\frac{31}{48}|T_3(i,j)| \leq |Q_j(i)| \leq \frac{65}{48}|T_3(i,j)| \text{ and } \frac{145}{288}|T_2(i+1,j)| \leq |Q_j(i+1)| \leq \frac{431}{288}|T_2(i+1,j)|.
$$
Now, by using (\ref{estimate}) we obtain that
$\displaystyle{\frac{|T_3(i,j)|}{|T_2(i+1,j)|} = \left|\frac{(b-1)b^d}{b^{d-2}-1} \right| \geq \frac{(1+q)q^2}{(1+q^{-(d-2)})} \geq \frac{192}{17}}.$
So,
$$
\frac{|Q_j(i)|}{|Q_j(i+1)|} \geq \frac{31/48|T_3(i,j)|}{431/288|T_2(i+1,j)|} \geq \frac{31/48}{431/288}\cdot \frac{192}{17} = \frac{35712}{7327} >1. 
$$
Third, suppose that $j\geq 4$. Use (\ref{consecutive-ratio}) and (\ref{useful-ratio}) together with  (\ref{estimate}) to get the following inequalities: $\displaystyle{\frac{|T_1(i,j)|}{|T_3(i,j)|} \leq \frac{(1+q^{-(d-2)})(1+q^3)}{(1-q^{-(j-2)})(1+q)q^5} \leq \frac{17}{128}}$ and $\displaystyle{\frac{|T_2(i,j)|}{|T_3(i,j)|} \leq \frac{(1+
q^{-(d-2)})(1+q^3)}{(1-q^{-(j-2)})(1+q)q^3} \leq \frac{17}{32}}$. Also, \newline $\displaystyle{\frac{|T_0(i+1,j)|}{|T_2(i+1,j)|} \leq \frac{(1+q^{-(d-1)})}{(1-q^{-(j-1)})q^3} \leq \frac{33}{224}}$ and $\displaystyle{\frac{|T_1(i+1,j)|}{|T_2(i+1,j)|} \leq \frac{(1+q^{-(d-1)})(q^2-1)}{(1-q^{-(j-1)})(1+q)q^2} \leq \frac{33}{112}}$. 
\newline By using these inequalities we obtain that
$$
\frac{43}{128}|T_3(i,j)| \leq |Q_j(i)| \leq \frac{213}{128}|T_3(i,j)| \text{ and } \frac{125}{224}|T_2(i+1,j)| \leq |Q_j(i+1)| \leq \frac{323}{224}|T_2(i+1,j)|.
$$
Now, by using (\ref{estimate}) we obtain that
$\displaystyle{\frac{|T_3(i,j)|}{|T_2(i+1,j)|} = \left|\frac{(b^{j-2}-1)b^{3-j+d}}{b^{d-2}-1} \right| \geq \frac{(1-q^{-(j-2)})q^3}{(1+q^{-(d-2)})} \geq \frac{96}{17}}.$
So
\begin{align*}
\frac{|Q_j(i)|}{|Q_j(i+1)|} \geq \frac{43/128|T_3(i,j)|}{323/224|T_2(i+1,j)|} \geq \frac{43/128}{323/224}\cdot \frac{96}{17} = \frac{7224}{5491} >1. \tag*{\qedhere} 
\end{align*}
\end{proof}
%------------------------------------------------

%------------------------------------------------
\begin{lemma}\label{lemma11}
Let $q \geq 2$ and $d\geq 6$. If $j\geq 5$ and $i = d-4$, then $ |Q_j(i)| > |Q_j(i+1)|$. 
\end{lemma}
\begin{proof} We have $h_{\max}(i,j) = 4$ and $h_{\max}(i+1,j) = 3$. We use (\ref{consecutive-ratio}) and (\ref{useful-ratio}) together with the estimate (\ref{estimate}) to obtain the following inequalities: $\displaystyle{\frac{|T_2(i,j)|}{|T_4(i,j)|} \leq \frac{(1+q^{-(d-3)})(q^4-1)(q^3+1)}{(1-q^{-(j-3)})(q^2-1)(q+1)q^7} \leq \frac{45}{256}}$ and $\displaystyle{\frac{|T_3(i,j)|}{|T_4(i,j)|} \leq \frac{(1+ q^{-(d-3)})(1-q^{-4})}{(1-q^{-(j-3)})(q+1)} \leq \frac{15}{32}}$. Also,  $\displaystyle{\frac{|T_1(i+1,j)|}{|T_3(i+1,j)|} \leq \frac{(1+q^{-(d-2)})(q^3+1)q^{-5}}{(1-q^{-(j-2)})(q+1)} \leq \frac{51}{448}}$ and $\displaystyle{\frac{|T_2(i+1,j)|}{|T_3(i+1,j)|} \leq \frac{(1+q^{-(d-2)})(q^3+1)}{(1-q^{-(j-2)})(q+1)q^3} \leq \frac{51}{112}}$. By using these inequalities we obtain that
$$
\frac{91}{256}|T_4(i,j)| \leq |Q_j(i)| \leq \frac{421}{256}|T_4(i,j)| \text{ and } \frac{193}{448}|T_3(i+1,j)| \leq |Q_j(i+1)| \leq \frac{703}{448}|T_3(i+1,j)|.
$$
Now, by using (\ref{estimate}) we obtain that
$\displaystyle{\frac{|T_4(i,j)|}{|T_3(i+1,j)|} = \left|\frac{(b^{j-3}-1)b^{4-j+d}}{b^{d-3}-1} \right| \geq \frac{(1-q^{-(j-3)})q^4}{(1+q^{-(d-3)})} \geq \frac{32}{3}}.$
So
\begin{align*}
\frac{|Q_j(i)|}{|Q_j(i+1)|} \geq \frac{91/256|T_4(i,j)|}{703/448|T_3(i+1,j)|} \geq \frac{91/256}{703/448}\cdot \frac{32}{3} = \frac{5096}{2109} >1. \tag*{\qedhere} 
\end{align*}
\end{proof}

\begin{theorem}\label{HermitianTheorem}
Let $q \geq 2$, $d\geq 6$, and $j\geq 1$.
\begin{enumerate}[(i)]
    \item Then $|Q_j(i)| > |Q_j(i+1)|$ for $0 \leq i\leq d-1$ except if $q=2$, $i =d-1$, and $j$ is equal to $d$ or is even.
    \item If $j$ is odd, then $Q_j(1) \leq Q_j(i)$ for $0 \leq i \leq d$.
    \item If $j$ is even, $j\geq 2$, then $Q_j(d-j+2) \leq Q_j(i)$ for $0\leq i \leq d$.
    \item Then $|Q_j(1)| > |Q_j(i)|$ for $2 \leq i \leq d$.
\end{enumerate}
\end{theorem}
\begin{proof}
\begin{enumerate}[(i)]
    \item It follows by Lemmas \ref{lemma3}, \ref{lemma7}, \ref{lemma8}, \ref{lemma9}, \ref{lemma10}, \ref{lemma11}, and the fact that $Q_j(0)$ is the largest eigenvalue of $Q_q(d,j)$.  Moreover, it imply that for the next two parts we only have to find the smallest value of $i$ such that $Q_j(i)$ is negative. By Lemmas \ref{lemma3}, \ref{threeparts}, and \ref{lemma9} the sign of $Q_j(i)$ is same as that of $T_{h_{\max}}(i,j)$. Note that the sign of $T_{h_{\max}}(i,j)$ is $(-1)^{ij+\binom{j-h_{\max}}{2}}$.
    \item If $j$ is odd, then smallest such value of $i$ is equal to 1. Thus, $Q_j(1) \leq Q_j(i)$ for $0\leq i \leq d$.
    \item If $j$ is even, then smallest such value of $i$ is equal to $d-j+2$. Thus, $Q_j(d-j+2) \leq Q_j(i)$ for $0 \leq i \leq d$. 
    \item Part (i) together with the second assertion of Lemma \ref{lemma8} imply that $|Q_j(1)| > |Q_j(i)|$ for $2 \leq i \leq d$. \qedhere
\end{enumerate}
\end{proof}

\section{Final remarks}

In this paper, we proved some conjectures from \cite{BCIM} related to the smallest eigenvalues of Grassmann graphs, Billinear forms graphs and Hermitian forms graphs. Let $q \geq 2$, $d\geq 1$ be integers. Let $Q$ be a set of size $q$. The Hamming scheme $H(d,q)$ is the association scheme with vertex set $Q^d$, and as relation the Hamming distance. For $0\leq j \leq d$
the vertex set of the graph $H(d,q,j)$ is $Q^d$ and two words in $Q^d$ are adjacent if their Hamming distance is $j$. The eigenvalues of $H(d,q,j)$ are $P_{ij} = K_j(i)$ ($0\leq i \leq d$), where 
\begin{align*}
    K_j(i) = \sum_{h=0}^{j}(-1)^h (q-1)^{j-h}\binom{i}{h}\binom{d-i}{j-h}.
\end{align*}
The following conjecture was also posed in \cite{BCIM} (see Conjecture 2.11) and is still open.
\begin{conjecture}\label{conjham}
If $H(d,q,j)$ is connected, it has more than $d/2$ distinct eigenvalues. 
\end{conjecture} 
\section*{Acknowledgments} We thank Ferdinand Ihringer for his feedback regarding this paper.

%-----References-----

\end{document}